\documentclass[11pt,reqno]{amsart}
\usepackage[T1]{fontenc}
\usepackage{lmodern}
\usepackage{amsmath,amssymb,amsthm,mathtools}
\usepackage{microtype}
\usepackage[letterpaper,textwidth=6.25in,textheight=9in,centering]{geometry}
\usepackage{enumitem}
\usepackage{xcolor}
\usepackage[colorlinks=true,linkcolor=black,citecolor=black,urlcolor=blue!45!black]{hyperref}
\newtheorem{theorem}{Theorem}[section]
\newtheorem{proposition}[theorem]{Proposition}
\newtheorem{lemma}[theorem]{Lemma}

\theoremstyle{definition}
\newtheorem{definition}[theorem]{Definition}
\newtheorem{example}[theorem]{Example}

\theoremstyle{remark}
\newtheorem{remark}[theorem]{Remark}
\numberwithin{equation}{section}
\DeclareMathOperator{\GL}{GL}
\DeclareMathOperator{\SL}{SL}
\DeclareMathOperator{\GA}{GA}
\DeclareMathOperator{\TA}{TA}
\DeclareMathOperator{\EA}{EA}
\DeclareMathOperator{\Spec}{Spec}
\DeclareMathOperator{\Frac}{Frac}
\DeclareMathOperator{\rank}{rank}
\DeclareMathOperator{\tr}{tr}

\DeclareMathOperator{\diag}{diag}
\DeclareMathOperator{\id}{id}

\newcommand{\Ncal}{\mathcal N}
\newcommand{\Dcal}{\mathcal D}

\newcommand{\Kbar}{\overline K}

\newcommand{\dd}{\mathrm d}
\newcommand{\trans}[1]{\operatorname{Trans}_{#1}}
\setlist[enumerate,1]{label=\textup{(\roman*)},leftmargin=2.2em,itemsep=3pt}
\allowdisplaybreaks[2]
\hypersetup{pdftitle={Nilpotent Jacobian maps in dimension three and stable tameness in block extensions},pdfauthor={Jie Wang}}

\title[Nilpotent Jacobian maps and stable tameness]{Nilpotent Jacobian maps in dimension three\protect\\ and stable tameness in block extensions}
\author{Jie Wang}
\address{Jie Wang, State Key Laboratory of Mathematical Sciences, Academy of Mathematics and Systems Science, Chinese Academy of Sciences, Beijing, China}
\email{wangjie212@amss.ac.cn}
\author{Dan Yan}
\address{Dan Yan, MOE-LCSM, School of Mathematics and Statistics,
Hunan Normal University, Changsha, China}
\email{yan-dan-hi@163.com}
\date{September 11, 2026}
\begin{document}
\begin{abstract}
We study polynomial maps in three variables with nilpotent Jacobian over a field of characteristic zero. The proposed classification reduces maps with linearly independent components to a family determined by a univariate polynomial evaluated at a quadratic coordinate. The geometric part of the argument produces two algebraically dependent constant linear combinations of the components. A derivation argument then yields the normal form over the original field. We obtain explicit polynomial inverses and tame factorizations. We then study higher-dimensional maps in which all but the last three components depend on three variables. For this class, we give separate normal forms for the two three-variable blocks and deduce stable tameness from the residue-field criterion of Berson, van den Essen, and Wright.
\end{abstract}
\maketitle
\section{Introduction}\label{sec:introduction}

Let $K$ be a field of characteristic zero. Keller's Jacobian conjecture asserts that a polynomial map $F\colon K^n\to K^n$ with $\det JF\in K^*$ has a polynomial inverse. Introduced by Keller~\cite{Keller}, the problem connects the geometry of affine space with the structure of polynomial algebras. For background on polynomial automorphisms, see~\cite{EssenBook}; for a recent account of the topology of polynomial maps, their nonproperness loci, and connections with the Jacobian conjecture, see El Hilany~\cite{ElHilany2025}.

The status of the general conjecture changed in July 2026 with Alp\"oge's three-dimensional counterexample. Tao~\cite{Tao2026} gives a geometric exposition of the example, while Gao~\cite{Gao2026} develops the tangent-sweep construction and its higher-dimensional extensions. The original map has constant Jacobian determinant $-2$ and identifies three distinct points. These developments show that the general Keller condition does not imply global invertibility in dimension at least three. They also require a distinction between that condition and the more restrictive nilpotent-Jacobian hypotheses studied here.

The reductions of Bass, Connell, and Wright~\cite{BCW} and Yagzhev~\cite{Yagzhev} reduce the general conjecture to maps $F=\id+H$ with $H$ cubic homogeneous and $JH$ nilpotent, with the dimension allowed to vary. These reductions can increase the dimension. Consequently, counterexamples to the general conjecture yield counterexamples within the reduced class in some dimension, but this argument does not determine the three-variable case with $JH$ nilpotent.

Recent work also examines the algebraic constraints imposed by the Jacobian condition. Lee and Li~\cite{LeeLi2024} use generalized Magnus expansions to obtain restrictions on the Newton polygons of inner polynomials associated with planar Jacobian pairs. Ram\'irez and Valqui~\cite{RamirezValqui2025} compute a Gr\"obner basis and describe the solution set of a polynomial system arising in a Laurent-series formulation of the planar problem. In a different direction, Hamada, Kato, and Komiya~\cite{HamadaKatoKomiya2025} relate a Tate-algebra version of the conjecture to the polynomial version under a Hausdorff adic-topology hypothesis. These approaches complement the study of normal forms, generic fibers, and explicit inverses pursued below.

Nilpotence also suggests a structural question. If $H(0)=0$ and $JH$ is nilpotent, must the components of $H$ be linearly dependent over $K$? The answer is affirmative in dimension two and, in any dimension, when the Jacobian has rank at most one; see~\cite{BCW,deBondtYan2014}. For homogeneous maps in dimension three, de Bondt and van den Essen proved the stronger conclusion of linear triangularizability~\cite{deBondtEssen}. The general dependence problem has counterexamples, including homogeneous counterexamples in higher dimensions~\cite{deBondt2006}. Thus a classification in dimension three must account for maps with linearly independent components.

Several classifications treat the case of independent components under additional hypotheses. Chamberland and van den Essen~\cite{ChamberlandEssen} considered maps
\[
 H=\bigl(u(x,y),v(x,y,z),h(u(x,y),v(x,y,z))\bigr).
\]
Yan and Tang~\cite{YanTang} studied restrictions on the degree in $z$, and Yan and de Bondt~\cite{YanBondt2019} treated related forms in arbitrary dimension. Classifications under further restrictions on variable dependence appear in~\cite{CastanedaEssen,Yan2022}. The manuscript of He and Yan~\cite{HeYan} studies maps $(u(x,y),v(x,y,z),h(x,y,z))$ under the inequality $\deg_zv\geq2\deg_zh$, and also treats the case $\deg_zh\leq3$. Its Problem~4.1 asks whether a common normal form remains valid without these restrictions.

This paper develops a classification of normalized three-variable maps with independent components and nilpotent Jacobian, and studies its consequences for tame and stably tame automorphisms. The argument is based on the family
\begin{equation}\label{eq:canonical-intro}
 \Ncal_P(x,y,z)=\bigl(P(y+x^2),\ z-2xP(y+x^2),\ P(y+x^2)^2\bigr),
 \qquad P\in K[t]\setminus K,\quad P(0)=0.
\end{equation}
Its Jacobian is nilpotent, and its components are linearly independent. Moreover, if
\[
 (X,Y,Z)=(x,y,z)+s\Ncal_P(x,y,z),
\]
then the identity
\begin{equation}\label{eq:intro-invariant}
 Y+X^2-sZ=y+x^2
\end{equation}
immediately gives a polynomial inverse. The main task is to reduce an arbitrary normalized map with independent components to this family.

The argument has two parts. First, Theorem~\ref{thm:dependent-pair} gives independent constant linear forms $\ell_1,\ell_2$ for which $\ell_1(H)$ and $\ell_2(H)$ are algebraically dependent. To find them, we pass from the image field to its relative algebraic closure in $K(x,y,z)$. Over this field, the generic curve is contained in an affine plane and satisfies an equation of degree at most two. Homogeneous triangularization restricts its directions at infinity. We exclude the line case and the conic with two distinct points at infinity; the remaining conic yields the dependent pair. A quadratic relation among the components then yields descent to $K$.

Second, Proposition~\ref{prop:pair-normal} converts a dependent pair into~\eqref{eq:canonical-intro}. After writing the pair as $f(r),g(r)$, we use a rational derivation with a slice to control the third component. A unit calculation rules out nonlinear dependence on a second invariant. Two commuting affine locally nilpotent derivations then show that $r$ is a polynomial in a coordinate $y+Q(x,z)$, where $\deg Q\leq2$. The remaining nilpotence identity forces the quadratic part of $Q$ to have rank one. Constant linear changes of coordinates complete the normalization. These changes are defined over $K$ and do not require the extraction of roots.

The resulting classification is stated in Theorem~\ref{thm:classification}. It leads to Theorem~\ref{thm1}: for $H\in K[x,y,z]^3$ with nilpotent Jacobian, the map $\id+H$ is tame. In the independent case we give an elementary factorization valid for the whole family $\id+sH$ over $K[s]$. In the dependent case, a constant linear conjugation makes one component vanish. A primitive relation over $K[z]$ then produces a polynomial change of basis in the remaining two coordinates. This gives a tame factorization over $K[s]$ as well.

We next consider maps in dimension $n\geq6$ of the form
\begin{equation}\label{eq:intro-block}
 \widetilde H=\bigl(H_1(x_1,x_2,x_3),\ldots,H_{n-3}(x_1,x_2,x_3),
 H_{n-2}(x),H_{n-1}(x),H_n(x)\bigr).
\end{equation}
Two three-variable nilpotent Jacobians occur as diagonal blocks of $J\widetilde H$. Theorem~\ref{thm2} gives their normal forms, using the fraction field of the base ring for the last block. This operation must be distinguished from conjugation of the full map by a matrix with variable coefficients. The constant term of the last block with respect to its three variables must also be subtracted.

The stable-tameness conclusion does not require descent of these rational changes of coordinates. We first show that the last block defines an automorphism over the polynomial base ring. Every residue-field specialization is tame by Theorem~\ref{thm1}. The residue-field criterion of Berson, van den Essen, and Wright~\cite[Theorem~4.12]{BEW} then gives stable tameness over that ring. Together with the tame base map and Suslin's theorem on polynomial matrices, this proves Theorem~\ref{thm3} for~\eqref{eq:intro-block}.

Section~\ref{sec:preliminaries} introduces the notation and records the algebraic results used below. Section~\ref{sec:dependent-pair} contains the dependent-pair argument. Section~\ref{sec:classification} gives the classification and tame factorizations. Section~\ref{sec:blocks} treats the higher-dimensional extensions. We conclude in Section~\ref{sec:conclusion} with several further questions.

\section{Notation and preliminaries}\label{sec:preliminaries}

Throughout the paper, $K$ denotes a field of characteristic zero. Polynomial maps and tuples of variables are written as column vectors. For $H\in K[x_1,\ldots,x_n]^n$, we write
\[
 JH=\left(\frac{\partial H_i}{\partial x_j}\right)_{1\leq i,j\leq n}.
\]
Matrix ranks are computed over the corresponding rational function field. We write $\deg_{x_i}f$ for the degree in $x_i$ and $\deg f$ for total degree. The components of $H$ are \emph{linearly independent} if no nonzero constant linear form annihilates $H$. A polynomial map $H$ is called \emph{normalized} if $H(0)=0$. For a map over a polynomial coefficient ring, normalization in specified polynomial variables means that the constant term in those variables is zero in the coefficient ring.

\subsection{Polynomial automorphisms}
Let $R$ be a commutative ring with identity. We denote by $\GA_n(R)$ the group of polynomial automorphisms of $R[x_1,\ldots,x_n]$, represented as polynomial maps and composed by substitution. An elementary automorphism has the form
\[
 e_i(f)=(x_1,\ldots,x_{i-1},x_i+f,x_{i+1},\ldots,x_n),
 \qquad f\in R[x_1,\ldots,\widehat{x_i},\ldots,x_n].
\]
The subgroup generated by these maps is $\EA_n(R)$. Constant translations are products of elementary automorphisms.

\begin{definition}\label{def:tame}
The tame subgroup is
\[
 \TA_n(R)=\langle\GL_n(R),\EA_n(R)\rangle\subseteq\GA_n(R).
\]
An automorphism $F\in\GA_n(R)$ is \emph{tame over $R$} if $F\in\TA_n(R)$. It is \emph{stably tame over $R$} if, for some $m\geq0$,
\[
 F^{[m]}=(F_1,\ldots,F_n,x_{n+1},\ldots,x_{n+m})\in\TA_{n+m}(R).
\]
\end{definition}

Equivalently, the tame group is generated by affine and elementary automorphisms. In particular, a triangular automorphism with unit diagonal coefficients is tame. When $R$ is itself a polynomial ring, tameness over $R$ treats the variables of the coefficient ring as coefficients. We use this distinction in Section~\ref{sec:blocks}.

\subsection{Nilpotence and linear conjugation}
For a $3\times3$ matrix $N$ over a characteristic-zero domain, let $\sigma_2(N)$ denote the sum of its principal minors of order two. The characteristic polynomial yields
\begin{equation}\label{eq:nilpotence}
 N\text{ is nilpotent}
 \quad\Longleftrightarrow\quad
 \tr N=\sigma_2(N)=\det N=0.
\end{equation}
In this case $N^3=0$. The stronger identity $N^2=0$ is not assumed. If $JH$ is nilpotent, then $\det(I+sJH)=1$ for an indeterminate $s$.

For $T\in\GL_n(K)$, set $H^T=T\circ H\circ T^{-1}$. Then
\begin{equation}\label{eq:constant-conjugacy}
 JH^T(x)=T(JH)(T^{-1}x)T^{-1}.
\end{equation}
Thus nilpotence, rank, normalization at the origin, and linear independence of the components are preserved. Formula~\eqref{eq:constant-conjugacy} requires a constant matrix. Nonlinear coordinates below are used either as auxiliary polynomial coordinates or as explicit factors of $\id+sH$.

We use two known triangularization results. A normalized map whose nilpotent Jacobian has rank at most one has linearly dependent components; see~\cite{BCW} or~\cite[Theorems~2.3 and~3.4]{deBondtYan2014}. A homogeneous map in three variables with nilpotent Jacobian is linearly triangularizable~\cite{deBondtEssen}. Consequently, a nonzero homogeneous map $G\in K[x,y,z]^3$ of degree $d\geq1$ with nilpotent Jacobian can be put in the form
\begin{equation}\label{eq:homogeneous-form}
 G=(0,\gamma x^d,B(x,y)),
\end{equation}
where $B$ is homogeneous of degree $d$. For $d=1$, this is ordinary triangularization of a nilpotent matrix.

\subsection{Derivations and slices}
Let $k$ be a field and let $\mathcal A$ be a commutative $k$-algebra. A \emph{$k$-derivation} of $\mathcal A$ is a $k$-linear map $\delta\colon\mathcal A\to\mathcal A$ satisfying the Leibniz rule
\[
 \delta(fg)=\delta(f)g+f\delta(g)\qquad(f,g\in\mathcal A).
\]
Its \emph{ring of constants} is $\ker\delta=\{f\in\mathcal A:\delta(f)=0\}$; this is a field when $\mathcal A$ is a field. An element $s\in\mathcal A$ is a \emph{slice} for $\delta$ if $\delta(s)=1$. A derivation admitting such an element is called a \emph{derivation with a slice}. For example, $\partial_x$ on $k[x,y,z]$ or $k(x,y,z)$ has the slice $x$.
A derivation $\delta$ is \emph{locally nilpotent} if, for every $f\in\mathcal A$, there is an integer $j\geq1$ such that $\delta^j(f)=0$. Having a slice does not by itself imply local nilpotence. 

\subsection{Closed polynomials and relative constants}
Let $k$ be any field.
\begin{definition}\label{def:closed}
A nonconstant polynomial $r\in k[x_1,\ldots,x_n]$ is \emph{closed} if $k(r)$ is relatively algebraically closed in $k(x_1,\ldots,x_n)$.
\end{definition}

In characteristic zero, this is equivalent to $k[r]$ being integrally closed in the ambient polynomial ring~\cite[Lemma~3]{ArzhantsevPetravchuk}. We use the closed-generator theorem: a finite family of polynomials generating a field of transcendence degree one is contained in $k[r]$ for a closed polynomial $r$; see~\cite{ArzhantsevPetravchuk}. Indeed, choose a generative polynomial for one nonconstant member. Every other member belongs to $k(r)$ by relative algebraic closedness, and
\begin{equation}\label{eq:intersection}
 k(r)\cap k[x_1,\ldots,x_n]=k[r].
\end{equation}
For completeness, write a rational function as $A(r)/B(r)$ with coprime $A,B\in k[t]$. A B\'ezout identity for these polynomials remains valid after substitution. If the quotient is polynomial, $B(r)$ divides $A(r)$ in the ambient polynomial ring, so $B(r)$ is a unit. Hence $B$ is constant. This last argument does not require closedness.

The following lemma collects the field-theoretic facts used below and specifies the constant field in the generic-fiber argument.

\begin{lemma}\label{lem:field-facts}
The following statements hold in characteristic zero.
\begin{enumerate}
\item If $\mathcal K/E$ is finitely generated and $L$ is the relative algebraic closure of $E$ in $\mathcal K$, then $L/E$ is finite and $\mathcal K/L$ is regular. Every finitely generated $L$-subalgebra of $\mathcal K$ is geometrically integral.
\item A derivation of a field extends uniquely to a finite separable extension. The extensions of commuting derivations commute.
\item Suppose $\mathcal K/k(t_1,\ldots,t_m)$ is finite separable, and extend the coordinate derivations to $\mathcal K$. If $\bar{a}\in \mathcal K$ is killed by $\partial_{t_i}$ for $i\in I$, then $\bar{a}$ is algebraic over $k(t_j:j\notin I)$.
\end{enumerate}
\end{lemma}

\begin{proof}
For (i), choose a transcendence basis $t_1,\ldots,t_r$ of $\mathcal K/E$ and put $d=[\mathcal K:E(t_1,\ldots,t_r)]$. Every finite intermediate extension $L_0/E$ contained in $L$ satisfies
\[
 [L_0:E]=[L_0(t_1,\ldots,t_r):E(t_1,\ldots,t_r)]\leq d.
\]
A finite subextension of maximal degree therefore contains every element of $L$, which proves finiteness. If an irreducible polynomial over $L$ factors into monic factors over $\mathcal K$, the coefficients of those factors lie in $\mathcal K$ and are algebraic over $L$. They therefore belong to $L$, contradicting irreducibility. Since finite extensions in characteristic zero are separable and simple, $\mathcal K$ is linearly disjoint from every finite extension of $L$. Hence $\mathcal K\otimes_L\overline L$ is a domain. Tensoring an inclusion of an $L$-subalgebra into $\mathcal K$ preserves injectivity and proves geometric integrality.

For (ii), if $\bar{a}$ has minimal polynomial $P(T)=\sum_j c_jT^j$, an extension of $\delta$ must satisfy
\[
 \delta \bar{a}=-\frac{\sum_j(\delta c_j)\bar{a}^j}{P'(\bar{a})}.
\]
Separability makes the denominator nonzero and gives existence and uniqueness. The commutator of two extended derivations is a derivation that vanishes on the base field and is therefore zero.

For (iii), differentiate the monic minimal polynomial of $\bar{a}$ over $k(t_1,\ldots,t_m)$. If $\partial_{t_i}\bar{a}=0$, the differentiated polynomial has smaller degree and still annihilates $\bar{a}$, so it must be zero. Thus its coefficients are killed by every indicated coordinate derivation. Their common constants in the rational function field are $k(t_j:j\notin I)$, which proves the assertion.
\end{proof}

The base field is relatively algebraically closed in a purely transcendental extension. In particular, the common constants of all extended coordinate derivations on a subfield of $k(x,y,z)$ are $k$ whenever the preceding lemma applies.

\section{A dependent pair of components}\label{sec:dependent-pair}

The first step is to find two independent constant linear combinations of the components that are algebraically dependent. The following theorem, proved in the remainder of this section, establishes this reduction.

\begin{theorem}[Dependent pair]\label{thm:dependent-pair}
Let $H\in K[x,y,z]^3$ satisfy $H(0)=0$. Suppose that $JH$ is nilpotent and that the components of $H$ are linearly independent over $K$. Then there exist linearly independent forms $\ell_1,\ell_2\in(K^3)^*$ such that $\ell_1(H)$ and $\ell_2(H)$ are algebraically dependent over $K$.
\end{theorem}

We first work over an algebraically closed field $k$. Since $\det JH=0$, the Jacobian rank is at most two. The dependence result recalled in Section~\ref{sec:preliminaries} excludes rank at most one. Hence $\rank JH=2$. After a simultaneous permutation of source and target coordinates, write $H=(u,v,w)$ with $u,v$ algebraically independent. Set
\[
 \mathcal K=k(x,y,z),\qquad E=k(u,v,w),
\]
and let $L$ be the relative algebraic closure of $E$ in $\mathcal K$. Then $L/k(u,v)$ is finite separable. The derivations $\partial_u,\partial_v$ extend uniquely to $L$ and commute. Put
\begin{equation}\label{eq:surface-derivatives}
 p=w_u,\quad q=w_v,\qquad a=w_{uu},\quad b=w_{uv},\quad c=w_{vv}.
\end{equation}

\subsection{The embedded generic curve}
Consider the Jacobian derivation
\[
 D(f)=\det J(u,v,f),\qquad (C_1,C_2,C_3)=\nabla u\times\nabla v.
\]
By the differential criterion for algebraic dependence,
\begin{equation}\label{eq:kernel-D}
 \ker_{\mathcal K}D=L.
\end{equation}
Indeed, $D(f)=0$ is equivalent to $f$ being algebraic over $k(u,v)$, and $E/k(u,v)$ is algebraic.

Since $\dd w=p\,\dd u+q\,\dd v$, the principal-minor identity is
\begin{equation}\label{eq:plane-sigma}
 \sigma_2(JH)=C_3-pC_1-qC_2.
\end{equation}
The right-hand side is $D(z-px-qy)$, because $p,q\in L$. Thus
\begin{equation}\label{eq:plane-s}
 s=z-px-qy\in L.
\end{equation}
Differentiating this identity, and setting
\[
 A=ax+by+s_u,\qquad B=bx+cy+s_v,
\]
we obtain
\begin{equation}\label{eq:plane-differentials}
 Au_x+Bv_x=-p,\qquad Au_y+Bv_y=-q,\qquad Au_z+Bv_z=1.
\end{equation}
These equations imply
\[
 \tr JH=u_x+v_y+pu_z+qv_z=-AC_1-BC_2.
\]
The coefficients of $A,B$, viewed as affine expressions in $x,y$, lie in $L$. It follows that
\begin{equation}\label{eq:quadratic-tau}
 \tau=\tfrac12ax^2+bxy+\tfrac12cy^2+s_ux+s_vy\in L.
\end{equation}

\begin{proposition}\label{prop:line-conic}
Let $\mathcal C=\Spec L[x,y,z]$, where $L[x,y,z]$ denotes the subalgebra of $\mathcal K$ generated by the source coordinates. Then $\mathcal C$ is a geometrically integral embedded affine line or plane conic. In the conic case its points at infinity have directions
\begin{equation}\label{eq:infinity-directions}
 e=(\xi,\eta,p\xi+q\eta),\qquad a\xi^2+2b\xi\eta+c\eta^2=0,
\end{equation}
over $\overline L$.
\end{proposition}

\begin{proof}
This algebra has fraction field $\mathcal K$ and dimension one over $L$. It is geometrically integral by Lemma~\ref{lem:field-facts}. Equations~\eqref{eq:plane-s} and~\eqref{eq:quadratic-tau} place the curve in the plane $Z=pX+qY+s$ and on the quadratic
\[
 \Phi(X,Y)=\tfrac12aX^2+bXY+\tfrac12cY^2+s_uX+s_vY-\tau.
\]
The coefficients $a,b,c$ are not all zero. Otherwise, both coordinate derivations on $L$ annihilate $p,q$, so they belong to $k$. The same argument gives $w-pu-qv\in k$. Evaluating at the origin then contradicts linear independence.

The kernel of $L[X,Y]\to \mathcal K$, $X\mapsto x$, $Y\mapsto y$, is a height-one prime containing the nonzero polynomial $\Phi$. It is generated by an irreducible factor of $\Phi$, of degree one or two. If the degree is two, that factor is a scalar multiple of $\Phi$. Homogenizing the plane and quadratic equations gives~\eqref{eq:infinity-directions}.
\end{proof}

The direction space of the containing affine plane is
\begin{equation}\label{eq:tangent-plane}
 \mathcal T=\{(\xi,\eta,p\xi+q\eta):\xi,\eta\in L\}.
\end{equation}
After extending scalars to $\mathcal K$, this is the column space of $JH$. No nonzero constant linear form vanishes on $\mathcal T$, since such a form would satisfy $\ell JH=0$, and hence $\ell(H)=0$ by normalization.

The construction is compatible with constant conjugation. For $S\in\GL_3(k)$, regard $x'=Sx$ and $H'=SH$ as elements of the same field $\mathcal K$. Then $k(H')=E$, $L[x']=L[x]$, and the Jacobian in the new coordinates is $S(JH)S^{-1}$. The transformation $S$ acts on the embedded curve, its directions, and $\mathcal T$. We may therefore repeat the construction with any algebraically independent pair of the new components.

\subsection{Constant planes at infinity}
Let $H_d\neq0$ be the highest homogeneous part of $H$. Taking the highest homogeneous part of the identity $(JH)^3=0$ gives $(JH_d)^3=0$. By~\eqref{eq:homogeneous-form}, a constant conjugation puts $H_d$ in the form $(0,\gamma x^d,B(x,y))$. If $\gamma\neq0$, its common zero set is contained in $x=0$. If $\gamma=0$, the nonzero binary form $B$ factors into linear forms over $k$. In either case there are finitely many nonzero constant linear forms $m_1,\ldots,m_N$ such that
\begin{equation}\label{eq:finite-planes}
 H_d(e)=0\quad\Longrightarrow\quad\prod_{j=1}^N m_j(e)=0.
\end{equation}
This implication holds over every extension of $k$.

The polynomial $H_i(X,Y,Z)-H_i(x,y,z)$, with coefficients in $L$, vanishes on $\mathcal C$. Homogenizing to the common degree $d$ shows that every point at infinity of its projective closure satisfies $H_d(e)=0$. Therefore every direction at infinity lies in one of the constant planes in~\eqref{eq:finite-planes}.

\subsection{Excluding a line}
Suppose that $\mathcal C$ is a line. Choose a constant form $m$ that vanishes on its direction and use $m$ as the first row of a simultaneous change of coordinates. The new coordinate $x$ is constant on the geometric line, so it is algebraic over $L$. Relative algebraic closedness gives $x\in L$.

Put $k_0=k(x)$. As elements of $k_0[y,z]$, the components of $H$ generate a field of transcendence degree one over $k_0$. Choose a closed generator $r\in k_0[y,z]$. Then
\begin{equation}\label{eq:line-field}
 L=k_0(r).
\end{equation}
Indeed, $r$ is algebraic over $E(x)$ and thus belongs to $L$. Conversely, every element of $L$ is algebraic over $k_0(r)$, which is relatively algebraically closed in $\mathcal K$.

The kernel defining the curve in the $Y,Z$ coordinates is generated by $r(Y,Z)-r$. To see this, first use the prime ideal $(r(Y,Z)-R)$ in $k_0[R,Y,Z]$ and then localize at $k_0[R]\setminus\{0\}$. Because the curve is a line, $r$ has degree one in $Y,Z$. We can write
\[
 r=a(x)y+b(x)z+c(x),\qquad a(x),b(x),c(x)\in k(x).
\]
If $a(x)=0$, all components of $H$ are independent of $y$. Otherwise, replace the generator by $s=y-\rho(x)z$ and write
\[
 H=(f(x,s),g(x,s),h(x,s)),\qquad f,g,h\in k(x)[s].
\]
The trace equation, in the independent variables $x,s,z$, becomes
\[
 f_x+g_s-\rho h_s-\rho' zf_s=0.
\]
Consequently $\rho'f_s=0$. If $\rho'\neq0$, then the first component in the source coordinates belongs to $k(x)\cap k[x,y,z]=k[x]$. Its Jacobian row is $(f'(x),0,0)$, and nilpotence forces $f'=0$. Normalization gives $f=0$, contrary to independence. Thus $\rho\in k$.

In either case, $H$ is invariant under translation in a fixed nonzero source direction. A constant conjugation sends that direction to the $z$-axis, giving a map $(A(x,y),B(x,y),C(x,y))$. The planar Jacobian of $(A,B)$ is nilpotent. The planar dependence result recalled in Section~\ref{sec:preliminaries} gives a constant relation between $A,B$, again contradicting independence. Hence $\mathcal C$ is a conic.

\subsection{Excluding two points at infinity}
Suppose that the conic has two distinct directions $e_1,e_2$ over $\overline L$. Choose constant forms $m_1,m_2$ with $m_i(e_i)=0$. Their restrictions to $\mathcal T$ are nonzero. On $\mathcal T\otimes_L\overline L$, their kernels are the distinct lines spanned by $e_1,e_2$, so the restrictions are linearly independent over $k$. Hence $u=m_1(H)$ and $v=m_2(H)$ are algebraically independent. Use their restrictions to $\mathcal T$ as the first two rows of a simultaneous coordinate change and complete the third row.

The projected directions in the new $x,y$ plane are $[0:1]$ and $[1:0]$. Recomputing~\eqref{eq:infinity-directions} gives $a=c=0$ and $b\neq0$. Commutation of derivatives gives $b_u=b_v=0$, and therefore $b=\gamma\in k^*$. Integration twice in the function field, (3.1) gives
\begin{equation}\label{eq:bilinear-graph}
 w=\gamma uv+\alpha u+\beta v+\delta,\qquad \alpha,\beta,\delta\in k.
\end{equation}
The following lemma excludes this possibility.

\begin{lemma}\label{lem:bilinear-exclusion}
A polynomial map $(u,v,\gamma uv+\alpha u+\beta v+\delta)$ with $\gamma\neq0$ cannot have both a nilpotent Jacobian and Jacobian rank two.
\end{lemma}

\begin{proof}
Suppose first that $m=\deg_z u>0$ and $n=\deg_zv>0$. The trace equation is
\begin{equation}\label{eq:bilinear-trace}
 u_x+v_y+\gamma(uv)_z+\alpha u_z+\beta v_z=0.
\end{equation}
The degree of $(uv)_z$ forces $m+n-1\leq\max(m,n)$, so $\min(m,n)=1$. If $m>n=1$, let $A,B\in k[x,y]\setminus\{0\}$ be the leading coefficients in $z$. The coefficient of $z^m$ is
\[
 A_x+(m+1)\gamma AB=0.
\]
Comparing degrees in $x$ over $k(y)$ gives a contradiction. The case $n>m=1$ is symmetric. If $m=n=1$, we obtain $A_x+B_y+2\gamma AB=0$. The total degree of $AB$ exceeds that of either nonzero derivative, again yielding a contradiction.

After interchanging the first two coordinates if necessary, assume $u_z=0$. Using the trace equation to eliminate $v_y$ from the principal-minor identity gives
\begin{equation}\label{eq:bilinear-minors}
 u_x^2+u_y\bigl(v_x+(\gamma v+\alpha)v_z\bigr)=0.
\end{equation}
If $u_y=0$, then $u_x=0$, so $u$ is constant and the Jacobian rank is at most one. If $v_z=0$, the third column of the Jacobian is zero and its first two rows are dependent by $\sigma_2=0$; again the rank is at most one. Thus $u_yv_z\neq0$.

If $\deg_zv\geq2$, the term $vv_z$ in~\eqref{eq:bilinear-minors} has strictly larger $z$-degree than the other terms. Therefore $v=Cz+D$ with $C\neq0$. The coefficient of $z$ gives $C_x+\gamma C^2=0$, which is impossible by comparing degrees in $x$.
\end{proof}

\subsection{The remaining direction}
The conic therefore has a unique point at infinity, counted twice. Let $e$ be its direction and choose a nonzero constant form $m$ with $m(e)=0$. Use $m$ as the first row of a simultaneous coordinate change. The first component $u=m(H)$ is nonconstant. Choose a constant combination $v$ algebraically independent of $u$, complete the coordinate change, and repeat the construction of the curve.

Now $e_x=0$, and~\eqref{eq:tangent-plane} implies $e_y\neq0$. The double root of the binary quadratic in~\eqref{eq:infinity-directions} is $[0:1]$. Hence $a\neq0$ and $b=c=0$. Both coordinate derivatives of $w_v$ vanish, so $w_v=\mu\in k$. It follows that
\[
 \dd(w-\mu v)=w_u\,\dd u.
\]
We give the field-theoretic justification for the resulting algebraic dependence. The extension $L/k(u,v)$ is finite separable, and the extended coordinate derivation $\partial_v$ satisfies
\[
 \partial_v(w-\mu v)=w_v-\mu=0.
\]
By Lemma~\ref{lem:field-facts}(iii), $w-\mu v$ is therefore algebraic over $k(u)$. More explicitly, let $P(T)\in k(u,v)[T]$ be its monic minimal polynomial. Differentiating $P(w-\mu v)=0$ with respect to $v$ gives $(\partial_vP)(w-\mu v)=0$, where $\partial_vP$ denotes coefficientwise differentiation. Since $P$ is monic, $\partial_vP$ is either zero or has degree smaller than $P$. Minimality forces $\partial_vP=0$. The coefficients of $P$ thus belong to the constant field of $\partial_v$ on $k(u,v)$, which is $k(u)$. Clearing their denominators yields a nonzero polynomial relation over $k$ between $u$ and $w-\mu v$. Thus $u$ and $w-\mu v$ are algebraically dependent over $k$. The corresponding constant forms are independent. This proves Theorem~\ref{thm:dependent-pair} over an algebraically closed field.

\subsection{Descent and completion of the proof}
We use Proposition~\ref{prop:pair-normal}, which is proved in the next section over an arbitrary characteristic-zero field, assuming that a dependent pair is already defined over that field. Its proof does not use Theorem~\ref{thm:dependent-pair}.

Extend scalars from $K$ to $\Kbar$. Linear independence of the components is preserved, since it is a rank condition on their coefficient matrix. The preceding argument and Proposition~\ref{prop:pair-normal} give the canonical form over $\Kbar$. Consider the finite-dimensional relation space
\begin{equation}\label{eq:relation-space}
 V_K=\{R\in K[U,V,W]_{\leq2}:R(H_1,H_2,H_3)=0\}.
\end{equation}
It is the kernel of a linear map of finite-dimensional coefficient spaces over $K$, and therefore
\[
 V_K\otimes_K\Kbar=V_{\Kbar}.
\]
For $\Ncal_P$, the first two components are algebraically independent and the third is the square of the first. Its full ideal of relations is $(W-U^2)$, whose subspace of relations of degree at most two is one-dimensional. Constant linear target changes preserve total degree, and source changes preserve relations. Consequently $\dim_KV_K=1$.

Choose $0\neq R\in V_K$. Its constant term is zero because $H(0)=0$. Over $\Kbar$, its quadratic part has rank one. The rank is also one over $K$, so
\[
 R=\kappa\ell^2+\ell',\qquad \kappa\in K^*,\quad \ell,\ell'\in(K^3)^*.
\]
Here a symmetric rank-one matrix over $K$ can be written as a nonzero scalar times an outer square over $K$ by choosing a nonzero diagonal entry. This does not require a square root. The forms $\ell,\ell'$ are independent. Otherwise $\ell'=c\ell$ and
\[
 \ell(H)\bigl(\kappa\ell(H)+c\bigr)=0.
\]
The polynomial ring is a domain, so $\ell(H)$ is constant. Its value at the origin is zero, contradicting independence of the components. Finally,
\[
 \ell'(H)=-\kappa\ell(H)^2
\]
is the required algebraic dependence over $K$. This completes the proof of Theorem~\ref{thm:dependent-pair}.\qed

\section{Classification and tameness}\label{sec:classification}

We now pass from the dependent pair in Theorem~\ref{thm:dependent-pair} to a canonical form. The classification rests on a polynomial calculation over the original field.

\begin{theorem}[Classification]\label{thm:classification}
Let $H\in K[x,y,z]^3$ satisfy $H(0)=0$, and suppose that its components are linearly independent over $K$. Then $JH$ is nilpotent if and only if there exist $T\in\GL_3(K)$ and a nonconstant polynomial $P\in K[t]$, with $P(0)=0$, such that
\begin{equation}\label{eq:canonical}
 T\circ H\circ T^{-1}=\Ncal_P,
 \qquad
 \Ncal_P=\bigl(P(y+x^2),\ z-2xP(y+x^2),\ P(y+x^2)^2\bigr).
\end{equation}
\end{theorem}

The theorem describes a family of normal forms; it does not assert that $P$ is unique under linear conjugation. We prove the algebraic part separately for use in the descent in Section~\ref{sec:dependent-pair}.

\subsection{From a dependent pair to a quadratic coordinate}
\begin{proposition}\label{prop:pair-normal}
Let $H\in K[x,y,z]^3$ be normalized, with nilpotent Jacobian and linearly independent components. If two independent constant linear combinations of its components are algebraically dependent over $K$, then $H$ is linearly conjugate over $K$ to $\Ncal_P$ for some $P\in K[t]\setminus K$ with $P(0)=0$.
\end{proposition}

\begin{proof}
A constant conjugation and the closed-generator theorem give
\begin{equation}\label{eq:pair-form}
 H=(f(r),v,g(r)),\qquad r(0)=f(0)=g(0)=0,
\end{equation}
where $r\in K[x,y,z]$ is closed and $f,g\in K[t]$ are linearly independent. Their derivatives are likewise linearly independent, so
\begin{equation}\label{eq:a-ratio}
 a(t)=\frac{g'(t)}{f'(t)}\in K(t)\setminus K,\qquad a'(t)\neq0.
\end{equation}
Primes in this proof denote univariate derivatives.

Put $W=f'(r)r_x+g'(r)r_z$. The determinant vanishes because the first and third Jacobian rows are proportional. The other nilpotence equations are
\begin{equation}\label{eq:pair-nilpotence}
 W+v_y=0,\qquad Wv_y-r_y\bigl(f'(r)v_x+g'(r)v_z\bigr)=0.
\end{equation}
We claim that $r_y\neq0$. If $r_y=0$, the two equations give $W=v_y=0$. Reducing $W=0$ modulo $r-\lambda$ gives
\[
 r-\lambda\mid f'(\lambda)r_x+g'(\lambda)r_z.
\]
The polynomial on the right has degree less than $\deg r$, and hence is zero. Two choices of $\lambda\in K$ with independent vectors $(f'(\lambda),g'(\lambda))$ force $r_x=r_z=0$, a contradiction.

\smallskip\noindent\emph{A derivation with a slice.}
On $\mathcal K=K(x,y,z)$ define
\begin{equation}\label{eq:delta}
 \Delta=\partial_x+a(r)\partial_z+\frac{v_y}{f'(r)r_y}\partial_y.
\end{equation}
Equations~\eqref{eq:pair-nilpotence} give
\[
 \Delta x=1,\qquad \Delta r=\Delta v=0,\qquad \Delta z=a(r).
\]
Set $s=z-a(r)x$. Then $\Delta s=0$, and $\det J(x,r,s)=r_y\neq0$. Thus $\mathcal K/K(x,r,s)$ is finite separable. Extend $\partial_r,\partial_s$ to $\mathcal K$. They commute with $\Delta$, the extension of $\partial_x$ in these coordinates. If $M=\ker_{\mathcal K}\Delta$, then both preserve $M$.

Evaluating on $x,r,s$ and using uniqueness of extension gives
\begin{equation}\label{eq:dy-identity}
 \partial_y=r_y\bigl(\partial_r-a'(r)x\partial_s\bigr).
\end{equation}
Since $v\in M$,
\[
 \Delta y=\frac{v_r-a'xv_s}{f'},
\]
where $a',f'$ are evaluated at $r$. Subtracting a polynomial in $x$ and applying $\Delta$ gives
\begin{equation}\label{eq:y-polynomial}
 y=c_0+\frac{v_r}{f'}x-\frac{a'v_s}{2f'}x^2,\qquad c_0\in M.
\end{equation}
Together with $z=s+a(r)x$, this proves $K[x,y,z]\subseteq M[x]$. The variable $x$ is transcendental over $M$: differentiating a minimal polynomial would otherwise contradict $\Delta x=1$ and separability.

\smallskip\noindent\emph{A unit calculation.}
Apply~\eqref{eq:dy-identity} to $y$:
\begin{equation}\label{eq:unit-identity}
 1=r_y\bigl(\partial_r-a'x\partial_s\bigr)y.
\end{equation}
Both factors are in $M[x]$. For the second factor, use~\eqref{eq:y-polynomial} and the fact that the extended derivations preserve $M$. Both factors are therefore units of $M[x]$. The coefficient of $x^3$ in the second factor is
\[
 \frac{(a')^2}{2f'}v_{ss},
\]
and therefore $v_{ss}=0$. The elements $v_s$ and $v-sv_s$ are killed by $\Delta$ and $\partial_s$. Lemma~\ref{lem:field-facts} shows that they are algebraic over $K(r)$. Since $r$ is closed, they lie in $K(r)$. Consequently
\begin{equation}\label{eq:v-rational}
 v=\alpha(r)z+\beta(r)x+\gamma(r),\qquad \alpha,\beta,\gamma\in K(t).
\end{equation}
More precisely, $\beta=-a\alpha$, although this relation is not needed in the remaining reduction.

We next show that the coefficients in~\eqref{eq:v-rational} are polynomials. Write
\[
 d(r)v=\alpha_0(r)z+\beta_0(r)x+\gamma_0(r)
\]
with $d,\alpha_0,\beta_0,\gamma_0\in K[t]$ have greatest common divisor one. If $d$ is nonconstant, choose a root $\lambda$ over $\overline K$. Reduction modulo $r-\lambda$ shows that $r-\lambda$ divides
$\alpha_0(\lambda)z+\beta_0(\lambda)x+\gamma_0(\lambda)$. The divisor has positive $y$-degree, whereas the dividend is independent of $y$. The dividend is therefore zero. All four univariate polynomials then share the root $\lambda$, contradicting their B\'ezout identity. Thus
\begin{equation}\label{eq:v-polynomial}
 \alpha,\beta,\gamma\in K[t].
\end{equation}

\smallskip\noindent\emph{Commuting affine derivations.}
Substitute~\eqref{eq:v-rational} into the trace equation. We obtain
\begin{equation}\label{eq:affine-trace}
 f'(r)r_x+g'(r)r_z+
 \bigl(\alpha'(r)z+\beta'(r)x+\gamma'(r)\bigr)r_y=0.
\end{equation}
For $\lambda\in K$, let
\[
 E_\lambda=f'(\lambda)\partial_x+g'(\lambda)\partial_z+
 \bigl(\alpha'(\lambda)z+\beta'(\lambda)x+\gamma'(\lambda)\bigr)\partial_y.
\]
Reduction of~\eqref{eq:affine-trace} modulo $r-\lambda$ gives $r-\lambda\mid E_\lambda r$. Since $\deg E_\lambda r\leq\deg r$, we have $E_\lambda r=\mu_\lambda(r-\lambda)$ for some $\mu_\lambda\in K$. Assign weights $1,2,1$ to $x,y,z$. The derivation $E_\lambda$ strictly lowers weighted degree, so it is locally nilpotent. This forces $\mu_\lambda=0$, so $E_\lambda r=0$ for every $\lambda\in K$.

Choose two values of $\lambda$ for which $(f'(\lambda),g'(\lambda))$ are independent. Constant linear combinations of the corresponding derivations give
\[
 D_1=\partial_x+b_1(x,z)\partial_y,\qquad
 D_2=\partial_z+b_2(x,z)\partial_y,
\]
with $b_1,b_2$ affine and $D_1r=D_2r=0$. Their commutator is
\[
 [D_1,D_2]=(\partial_xb_2-\partial_zb_1)\partial_y.
\]
Its coefficient is constant. Since the commutator annihilates $r$ and $r_y\neq0$, its coefficient is zero. There is therefore a polynomial $Q\in K[x,z]$ with
\[
 \deg Q\leq2,\quad Q(0)=0,\qquad Q_x=-b_1,\quad Q_z=-b_2.
\]
In the auxiliary polynomial coordinates $(x,\rho,z)$, where $\rho=y+Q(x,z)$, the derivations become $\partial_x,\partial_z$. Their common kernel in the polynomial ring is $K[\rho]$, so $r=R(\rho)$ for a polynomial $R$ with $R(0)=0$. Absorb $R$ into the univariate polynomials in~\eqref{eq:pair-form} and~\eqref{eq:v-rational}.

Write
\[
 Q=\tfrac12Ax^2+Bxz+\tfrac12Cz^2+Dx+Ez.
\]
Expressing the trace equation in the coordinates $x,\rho,z$ and comparing coefficients gives
\[
 \beta'=-Af'-Bg',\qquad \alpha'=-Bf'-Cg',\qquad \gamma'=-Df'-Eg'.
\]
Integrating and using $H(0)=0$, we obtain
\begin{equation}\label{eq:quadratic-form}
 H=\bigl(f(\rho),-Q_xf(\rho)-Q_zg(\rho)+\Lambda(x,z),g(\rho)\bigr),
 \qquad \rho=y+Q(x,z),
\end{equation}
where $\Lambda$ is a homogeneous linear polynomial and $f(0)=g(0)=0$. If $Q_2$ is the quadratic homogeneous part of $Q$, expansion of the principal minors gives
\begin{align}\label{eq:quadratic-relation-derivative}
 \sigma_2(JH)
 &=f'(\rho)\bigl(Q_{xx}f(\rho)+Q_{xz}g(\rho)-\Lambda_x\bigr)\notag\\
 &\quad+g'(\rho)\bigl(Q_{xz}f(\rho)+Q_{zz}g(\rho)-\Lambda_z\bigr)\\
 &=\left.\frac{\dd}{\dd t}
 \bigl(Q_2(f(t),g(t))-\Lambda(f(t),g(t))\bigr)\right|_{t=\rho}.\notag
\end{align}
Substitution $t\mapsto \rho$ is injective. The polynomial being differentiated is constant and vanishes at zero. Thus
\begin{equation}\label{eq:quadratic-relation}
 Q_2(f(t),g(t))=\Lambda(f(t),g(t)).
\end{equation}

\smallskip\noindent\emph{Constant linear normalization.}
The shear $(x,y,z)\mapsto(x,y+Dx+Ez,z)$ removes the linear part of $Q$ from~\eqref{eq:quadratic-form}; the corresponding target change cancels the terms $-Df-Eg$. We may therefore assume $Q=Q_2$. If $Q_2=0$, relation~\eqref{eq:quadratic-relation} and independence of $f,g$ force $\Lambda=0$. The second component would then vanish, contradicting independence.

The binary quadratic form $Q_2$ cannot have rank two. Over $\overline K$, such a form is a product of independent linear forms. After evaluation at $(f,g)$, relation~\eqref{eq:quadratic-relation} would give
\[
  \bar{F}G=\alpha \bar{F}+\beta G,\qquad \bar{F}(0)=G(0)=0,
\]
with $\bar{F},G$ linearly independent over $ \overline K$. But
$(\bar{F}-\beta)(G-\alpha)=\alpha\beta$. If $\alpha\beta\neq0$, both factors are units of $\overline K[t]$. If $\alpha\beta=0$, one factor is zero. In either case, $\bar{F}$ or $G$ is constant, contradicting independence.

Thus $Q_2$ has rank one and can be written over $K$ as $\kappa L(x,z)^2$, with $\kappa\in K^*$ and $L$ a nonzero linear form. After a simultaneous linear change of $x,z$, write
\[
 Q_2=\kappa x^2,\qquad \Lambda=\alpha x+\beta z.
\]
Then $\kappa f^2=\alpha f+\beta g$. Necessarily $\beta\neq0$, and~\eqref{eq:quadratic-form} becomes
\begin{equation}\label{eq:precanonical}
 H=\left(f(\rho),\ \beta z-2\kappa xf(\rho)+\alpha x,
 \ \frac{\kappa}{\beta}f(\rho)^2-\frac{\alpha}{\beta}f(\rho)\right),
 \qquad \rho=y+\kappa x^2.
\end{equation}
The shear $z'=z+(\alpha/\beta)x$ removes the terms involving $\alpha$. Finally set
\[
 x'=x,\qquad y'=y/\kappa,\qquad z''=(\beta/\kappa)z'.
\]
With $P(t)=f(\kappa t)$, the transformed map is $\Ncal_P(x',y',z'')$. All these changes are defined over $K$, and $P(0)=0$.
\end{proof}

\begin{proof}[Proof of Theorem~\ref{thm:classification}]
The forward implication follows from Theorem~\ref{thm:dependent-pair} and Proposition~\ref{prop:pair-normal}. For the converse, put $\rho=y+x^2$, $u=P(\rho)$, and $\bar{q}=P'(\rho)$. Then
\begin{equation}\label{eq:canonical-J}
 J\Ncal_P=
 \begin{pmatrix}
 2x\bar{q}&\bar{q}&0\\
 -2u-4x^2\bar{q}&-2x\bar{q}&1\\
 4xu\bar{q}&2u\bar{q}&0
 \end{pmatrix}.
\end{equation}
The trace and determinant vanish. The principal minors of order two are $2u\bar{q},0,-2u\bar{q}$, so nilpotence follows from~\eqref{eq:nilpotence}. The first and third components are independent because $u$ is nonconstant, and the second component has coefficient one in $z$. Hence all three are independent.
\end{proof}

\begin{remark}
Theorem~\ref{thm:classification} provides an affirmative answer to Problem~3.8 in~\cite{YanTang} and Problem~4.1 in \cite{HeYan}.
\end{remark}

\subsection{The dependent case over a polynomial coefficient ring}
To prove tameness, we also need a normal form for maps with dependent components. The coefficient ring is essential to the following elementary argument.

\begin{lemma}\label{lem:planar-form}
Let $R=K[z]$ and let $h=(h_1,h_2)\in R[x,y]^2$ have nilpotent planar Jacobian $J_{x,y}h$. There are $a_1,a_2,c_1,c_2\in R$ and $f\in R[t]$ such that
\begin{equation}\label{eq:dependent-form}
 h_1=a_2f(a_1x+a_2y)+c_1,\qquad
 h_2=-a_1f(a_1x+a_2y)+c_2,
\end{equation}
where $(a_1,a_2)=R$ and $f(0)=0$.
\end{lemma}

\begin{proof}
Set $c_i=h_i(0,0)$ and $\bar h=h-c$. Over $K(z)$, the normalized planar map $\bar h$ has nilpotent Jacobian and therefore dependent components. Clearing the denominators in a relation and dividing by the greatest common divisor gives
\[
 a_1\bar h_1+a_2\bar h_2=0,\qquad a_1,a_2\in R,\quad(a_1,a_2)=R.
\]
If $\bar h=0$, take $(a_1,a_2)=(1,0)$ and $f=0$. Otherwise, the primitive relation gives
$\bar h=(a_2q,-a_1q)$ for $q\in R[x,y]$. For example, if $u a_1+v a_2=1$, then $q=v\bar h_1-u\bar h_2$.

Choose $\alpha,\beta\in R$ with $\alpha a_1-\beta a_2=1$, and put
\begin{equation}\label{eq:unimodular-M}
 M(z)=\begin{pmatrix}a_1&a_2\\ \beta&\alpha\end{pmatrix}\in\SL_2(R),
 \qquad (s,t)=(a_1x+a_2y,\beta x+\alpha y).
\end{equation}
The trace equation is $a_2q_x-a_1q_y=0$. In the coordinates $s,t$, the derivation $a_2\partial_x-a_1\partial_y$ is $-\partial_t$. Hence $q=f(s)$ for $f\in R[s]$. Since $\bar h(0,0)=0$ and the pair is primitive, $f(0)=0$.
\end{proof}

\begin{remark}\label{rem:coefficient-ring}
Lemma~\ref{lem:planar-form} is Theorem 7.2.25 in \cite{EssenBook}. We give the proof here for completeness. In Lemma~\ref{lem:planar-form}, $f$ is a polynomial in one variable with coefficients in $K[z]$. The proof does not justify replacing $K[z][t]$ by $K[t]$. This distinction also applies to the dependent normal forms in Section~\ref{sec:blocks}.
\end{remark}

\subsection{Tame factorizations and polynomial inverses}
\begin{theorem}[Tameness]\label{thm1}
Let $H\in K[x,y,z]^3$ have nilpotent Jacobian. Then $F=\id+H$ is tame over $K$. More generally, for an indeterminate $s$,
\[
 \id+sH\in\TA_3(K[s]).
\]
\end{theorem}

\begin{proof}
Write $H=\bar H+c$, where $c=H(0)$. Then
\begin{equation}\label{eq:translation-normalize}
 \id+sH=\trans{sc}\circ(\id+s\bar H).
\end{equation}
It suffices to treat $\bar H$. We distinguish the cases of linearly independent and linearly dependent components after normalization.

Suppose first that the components of $\bar H$ are independent. By Theorem~\ref{thm:classification}, a constant conjugation transforms the map into $\Ncal_P$. Define
\[
 T_Q(x,y,z)=(x,y+Q(x,z),z),\qquad
 S_s(x,y,z)=(x+sf(y),y,z+sg(y)).
\]
For the quadratic-coordinate form~\eqref{eq:quadratic-form}, relation~\eqref{eq:quadratic-relation} gives
\begin{equation}\label{eq:tame-factorization}
 \id+sH=T_{Q-s\Lambda}^{-1}\circ S_s\circ T_Q.
\end{equation}
Indeed, the second component of the right-hand side is
\begin{align*}
 \rho-Q(x+sf(\rho),z+sg(\rho))+s\Lambda(x+sf(\rho),z+sg(\rho))
 &=y+s\bigl(-Q_xf(\rho)-Q_zg(\rho)+\Lambda(x,z)\bigr)\\
 &\quad+s^2\bigl(\Lambda(f(\rho),g(\rho))-Q_2(f(\rho),g(\rho))\bigr).
\end{align*}
The last line vanishes. The other two components follow directly from the composition. Each $T$ is elementary and $S_s$ is a product of two commuting elementary maps. In the canonical case one takes $Q=x^2$, $\Lambda=z$, $f=P$, and $g=P^2$.

Suppose next that the components of $\bar H$ are dependent. A constant conjugation gives $\bar H=(h_1,h_2,0)$. Apply Lemma~\ref{lem:planar-form} and let
\[
 S(x,y,z)=(M(z)(x,y)^{\mathsf T},z),
\]
where $M$ is given by~\eqref{eq:unimodular-M}. Since the third coordinate is fixed,
\begin{equation}\label{eq:dependent-tame-conjugate}
 S\circ(\id+s\bar H)\circ S^{-1}
 =\bigl(x+s(a_1c_1+a_2c_2),\
 y-sf(x)+s(\beta c_1+\alpha c_2),\ z\bigr).
\end{equation}
Here $f(x)\in K[z][x]$. This map is triangular and therefore tame over $K[s]$. The Euclidean algorithm over $K[z]$ expresses $M\in\SL_2(K[z])$ as a product of elementary matrices. Each factor lifts to an elementary automorphism of $K[s][x,y,z]$, so $S$ and $S^{-1}$ are tame. Undoing the constant conjugation and~\eqref{eq:translation-normalize} proves the theorem.
\end{proof}

The factorization~\eqref{eq:tame-factorization} also gives the inverse before the final linear normalization. For the quadratic-coordinate form~\eqref{eq:quadratic-form}, set
\begin{equation}\label{eq:quadratic-inverse}
 \begin{aligned}
 \pi&=Y+Q(X,Z)-s\Lambda(X,Z),\\
 x&=X-sf(\pi),\qquad z=Z-sg(\pi),\qquad y=\pi-Q(x,z).
 \end{aligned}
\end{equation}
Indeed, the inverse factorization is $T_Q^{-1}\circ S_{-s}\circ T_{Q-s\Lambda}$, which gives the displayed expressions. Thus~\eqref{eq:quadratic-inverse} is a two-sided polynomial inverse over $K[s]$.

In the canonical case, if
\[
 (X,Y,Z)=(x+sP(\rho),y+s(z-2xP(\rho)),z+sP(\rho)^2),\qquad \rho=y+x^2,
\]
put $\pi=Y+X^2-sZ$. Then
\begin{equation}\label{eq:canonical-inverse}
 (\id+s\Ncal_P)^{-1}(X,Y,Z)
 =\bigl(X-sP(\pi),\ \pi-(X-sP(\pi))^2,\ Z-sP(\pi)^2\bigr).
\end{equation}
Equation~\eqref{eq:intro-invariant} proves both compositions. The construction is polynomial in $s$ and involves no division by the parameter.

In treating dependent components, we also retain the constant term with respect to the planar variables. For example, $H=(z,x,0)$ vanishes at the origin and has nilpotent Jacobian, but $H((x,y,z)+sH(x,y,z))\neq H(x,y,z)$. Its constant term as a planar map over $K[z]$ is $(z,0)$. Direct substitution gives
\[
 (\id+sH)^{-1}(X,Y,Z)=(X-sZ,\ Y-sX+s^2Z,\ Z),
\]
which illustrates why normalization at the origin does not replace normalization in the two polynomial variables over the coefficient ring.

\subsection{The generic curve of the canonical family}\label{subsec:canonical-curve}
The canonical family gives an explicit model for the relative constant field used in Section~\ref{sec:dependent-pair}. Put
\[
 \rho=y+x^2,\qquad u=P(\rho),\qquad v=z-2xu.
\]
Since the third component of $\Ncal_P$ is $u^2$, the image field and its relative algebraic closure in $\mathcal K=K(x,y,z)$ are
\begin{equation}\label{eq:canonical-fields}
 E=K(u,v),\qquad L=K(\rho,v),\qquad \mathcal K=L(x),\qquad [L:E]=\deg P.
\end{equation}
Indeed, $y=\rho-x^2$ and $z=v+2xP(\rho)$, so $\mathcal K=L(x)$. The field $L$ is relatively algebraically closed in this purely transcendental extension. Moreover, $L/E$ is finite of degree $\deg P$, the degree of the univariate extension $K(\rho)/K(P(\rho))$ after adjoining the independent variable $v$.

Over $L$, the embedded generic curve has the parametrization
\begin{equation}\label{eq:canonical-conic}
 x\longmapsto\bigl(x,\ \rho-x^2,\ v+2P(\rho)x\bigr).
\end{equation}
It is a parabola in the containing affine plane, with the unique direction at infinity $[0:1:0]$. If $\deg P>1$, the geometric generic fiber over $E$ has several components, corresponding to the distinct generic roots of $P(\rho)=u$. Passing to $L$ selects one geometrically integral component. This explains the use of the relative algebraic closure in the proof of Proposition~\ref{prop:line-conic}.

The source coordinate $\rho$ and the target relation $W-U^2=0$ have different roles. The source coordinate describes the conic fibers, whereas the target relation defines the one-dimensional relation space used in the descent in Section~\ref{sec:dependent-pair}.

\section{Block extensions and stable tameness}\label{sec:blocks}

We now consider higher-dimensional extensions, separating the base variables from the last three variables. Let $n\geq6$, put $m=n-6$, and write
\[
 \iota=(x_1,x_2,x_3),\qquad \varphi=(x_4,\ldots,x_{n-3}),\qquad
 \psi=(x_{n-2},x_{n-1},x_n).
\]
The tuple $\varphi$ is empty when $n=6$. Set $\zeta=(\iota,\varphi)$ and $R=K[\zeta]$, $L=\Frac R$. We consider maps of the form
\begin{equation}\label{eq:block-map}
 \widetilde H(\iota,\varphi,\psi)=\bigl(A(\iota),C(\iota),B(\zeta,\psi)\bigr),
\end{equation}
where $A\in K[\iota]^3$, $C\in K[\iota]^m$, and $B\in R[\psi]^3$.

For a field $K$, we use the notation
\begin{equation}\label{eq:D-family}
 \begin{aligned}
 \Dcal_{\vartheta_1,\vartheta_2,h,d_1,d_2}(r_1,r_2,r_3)
 =\bigl(&\vartheta_2h(\vartheta_1r_1+\vartheta_2r_2)+d_1,\\
       &-\vartheta_1h(\vartheta_1r_1+\vartheta_2r_2)+d_2,\ 0\bigr),
 \end{aligned}
\end{equation}
where $\vartheta_1,\vartheta_2,d_1,d_2\in K[r_3]$, $(\vartheta_1,\vartheta_2)=K[r_3]$, and $h\in K[r_3][t]$ with $h(0)=0$. We retain $\Ncal_P$ for the independent family~\eqref{eq:canonical}.

\subsection{Normal forms of the two blocks}
\begin{theorem}[Blockwise normal forms]\label{thm2}
Suppose that $\widetilde H$ has the form~\eqref{eq:block-map}, $\widetilde H(0)=0$, and $J\widetilde H$ is nilpotent. Then the following statements hold.
\begin{enumerate}
\item There is $S\in\GL_3(K)$ such that $S\circ A\circ S^{-1}$ is either a map of the form~\eqref{eq:D-family} over $K$, or $\Ncal_P$ for a nonconstant $P\in K[t]$ with $P(0)=0$.
\item Put $\chi(\zeta)=B(\zeta,0)$ and $B^0(\zeta,\psi)=B(\zeta,\psi)-\chi(\zeta)$. Viewed as a polynomial map in $\psi$ over $L$, the normalized map $B^0$ admits $T\in\GL_3(L)$ such that
\begin{equation}\label{eq:fiber-normal-form}
 T B^0(\zeta,T^{-1}\psi)=\Dcal_{\upsilon_1,\upsilon_2,f,\omega_1,\omega_2}(\psi)
 \quad\text{or}\quad \Ncal_Q(\psi).
\end{equation}
In the first case, $\upsilon_1,\upsilon_2,\omega_1,\omega_2\in L[\psi_3]$ and $f\in L[\psi_3][t]$, with $(\upsilon_1,\upsilon_2)=L[\psi_3]$ and $f(0)=0$. In the second case, $Q\in L[t]\setminus L$ and $Q(0)=0$.
\end{enumerate}
The constant conjugation $\widehat S=\diag(S,I_m,I_3)$ of the full map is
\begin{equation}\label{eq:whole-base-conjugation}
 \widehat S\circ\widetilde H\circ\widehat S^{-1}
 =\bigl(SA(S^{-1}\iota),\ C(S^{-1}\iota),\ B(S^{-1}\iota,\varphi,\psi)\bigr).
\end{equation}
The transformation in~\eqref{eq:fiber-normal-form} changes the three polynomial variables over $L$, while keeping the base parameters fixed.
\end{theorem}

\begin{proof}
The Jacobian is block lower triangular:
\begin{equation}\label{eq:block-jacobian}
 J\widetilde H=
 \begin{pmatrix}
 J_\iota A&0&0\\
 J_\iota C&0&0\\
 J_\iota B&J_\varphi B&J_\psi B
 \end{pmatrix}.
\end{equation}
Its diagonal blocks $J_\iota A$ and $J_\psi B$ are nilpotent. Since $A(0)=0$, Theorem~\ref{thm:classification} applies if the components of $A$ are independent. Otherwise, a constant conjugation makes its third component vanish, and Lemma~\ref{lem:planar-form}, with $\iota_3$ as the coefficient variable, gives~\eqref{eq:D-family}. This proves (i), and direct substitution gives~\eqref{eq:whole-base-conjugation}.

Over $L$, the map $B^0$ satisfies $B^0(0)=0$ and $J_\psi B^0=J_\psi B$. Again distinguish the cases of dependent and independent components, now over the characteristic-zero field $L$. The independent case follows from Theorem~\ref{thm:classification}. In the dependent case, a matrix in $\GL_3(L)$ makes one component zero, and Lemma~\ref{lem:planar-form} applies over $L[\psi_3]$. This proves (ii).
\end{proof}

\begin{remark}\label{rem:block-interpretation}
The theorem concerns the two three-variable diagonal blocks. A matrix with entries in $K(\zeta)$ is constant for differentiation in $\psi$, but need not be constant for the full Jacobian. Thus~\eqref{eq:fiber-normal-form} cannot be substituted into~\eqref{eq:constant-conjugacy} for the full map. In particular, the notation $T\circ\widetilde H\circ T^{-1}$ with $T\in\GL_n(K(\zeta))$ does not by itself define a constant linear conjugation of $\widetilde H$. Formula~\eqref{eq:whole-base-conjugation} also shows why the middle block changes when the first three source coordinates change.
\end{remark}

The subtraction of $\chi(\zeta)$ is necessary even when $\widetilde H(0)=0$.

\begin{example}\label{ex:fiber-constant}
Take $n=6$, $A=0$, and
\[
 B(\zeta,\psi)=\bigl(\varepsilon,\ \psi_3-2\psi_1\varepsilon,\ \varepsilon^2+\zeta_1\bigr),\qquad \varepsilon=\psi_2+\psi_1^2.
\]
The full map $(0,0,0,B)$ vanishes at the origin and has nilpotent Jacobian by~\eqref{eq:block-jacobian} and~\eqref{eq:canonical-J}. Over $L=K(\zeta_1,\zeta_2,\zeta_3)$, its last three components are linearly independent. Their relation ideal is
\[
 (W-U^2-\zeta_1)\subset L[U,V,W],
\]
because $\varepsilon$ and $\psi_3-2\psi_1\varepsilon$ are algebraically independent. In particular, every nonzero relation of total degree at most two has a nonzero constant term.

A dependent form~\eqref{eq:D-family} has a nonzero linear relation. A canonical form $\Ncal_Q$ has the quadratic relation $W-U^2=0$, whose constant term is zero, even if one allows $Q(0)\neq0$. Constant linear target changes preserve the constant term of a relation, and source changes preserve the relation ideal. Therefore $B$ is not linearly conjugate over $L$ to either displayed family. After subtracting $B(\zeta,0)=(0,0,\zeta_1)$, the resulting map is exactly $\Ncal_t$.
\end{example}

\subsection{Automorphisms over the base ring}
The normal form over the fraction field alone does not establish stable tameness over $R$. We first prove polynomial invertibility over $R$.

\begin{lemma}\label{lem:inverse-descent}
Let $R$ be a domain, $L=\Frac R$, and let $G\in R[\psi_1,\ldots,\psi_r]^r$ satisfy $G(0)=0$ and $JG(0)\in\GL_r(R)$. If $G\in\GA_r(L)$, then $G\in\GA_r(R)$.
\end{lemma}

\begin{proof}
The formal inverse theorem over a ring gives a unique inverse $Q\in R[[\psi]]^r$ with $Q(0)=0$. It can be constructed degree by degree: after the terms below degree $d$ are known, the degree-$d$ correction is obtained by multiplying the degree-$d$ error by $JG(0)^{-1}$. Thus all coefficients lie in $R$.

The polynomial inverse of $G$ over $L$ fixes the origin. As an element of $L[[\psi]]^r$, it equals $Q$ by uniqueness. Hence $Q$ has only finitely many nonzero coefficients and belongs to $R[\psi]^r$. Both inverse identities hold over $R$, either by the formal construction or by the injection $R[\psi]\hookrightarrow L[\psi]$.
\end{proof}

\begin{proposition}\label{prop:ring-nilpotent}
Let $R=K[\zeta_1,\ldots,\zeta_r]$, and let $B\in R[\psi_1,\psi_2,\psi_3]^3$ have nilpotent Jacobian $J_\psi B$. Then
\[
 \Theta(\psi)=\psi+B(\zeta,\psi)
\]
is a polynomial automorphism over $R$ and is stably tame over $R$.
\end{proposition}

\begin{proof}
Put $\chi=B(\zeta,0)$ and $G(\psi)=\psi+B(\zeta,\psi)-\chi$. Then $G(0)=0$. The matrix $N=J_\psi B(\zeta,0)$ satisfies $N^3=0$, so
\[
 JG(0)=I+N\in\GL_3(R),\qquad (I+N)^{-1}=I-N+N^2.
\]
Over $L=\Frac R$, Theorem~\ref{thm1} implies that $G$ is a polynomial automorphism. Lemma~\ref{lem:inverse-descent} gives $G\in\GA_3(R)$. Since $\Theta=\trans{\chi}\circ G$, we also have $\Theta\in\GA_3(R)$.

For every prime $\mathfrak p\in\Spec R$, let $\kappa(\mathfrak p)$ be the residue field. It contains $K$, so it has characteristic zero. Nilpotence is expressed by a polynomial matrix identity and is therefore preserved under specialization. Thus
\[
 \Theta_{\mathfrak p}=\psi+B_{\mathfrak p}(\psi)
 \in\TA_3(\kappa(\mathfrak p))
\]
by Theorem~\ref{thm1}. The ring $R$ is regular. The residue-field criterion~\cite[Theorem~4.12]{BEW} therefore implies that $\Theta$ is stably tame over $R$.
\end{proof}

\subsection{Stable tameness of the full map}
\begin{theorem}[Stable tameness of block extensions]\label{thm3}
Let $n\geq6$ and let $\widetilde H\in K[x_1,\ldots,x_n]^n$ have nilpotent Jacobian $J\widetilde H$ and satisfy
\[
 H_i\in K[x_1,x_2,x_3]\qquad(1\leq i\leq n-3).
\]
Then $F=\id+\widetilde H$ is a polynomial automorphism and is stably tame over $K$. In particular, this holds for the maps of Theorem~\ref{thm2}.
\end{theorem}

\begin{proof}
Write $F$ in block form:
\begin{equation}\label{eq:block-F}
 \begin{aligned}
 F(\zeta,\psi)&=\bigl(F_0(\zeta),\Theta_{\zeta}(\psi)\bigr),\\
 F_0(\iota,\varphi)&=(\iota+A(\iota),\varphi+C(\iota)),\qquad \Theta_{\zeta}(\psi)=\psi+B(\zeta,\psi).
 \end{aligned}
\end{equation}
The diagonal blocks in~\eqref{eq:block-jacobian} are nilpotent. By Theorem~\ref{thm1}, $g(\iota)=\iota+A(\iota)$ is tame over $K$. The base automorphism factors as
\[
 F_0=(\iota,\varphi+C(g^{-1}(\iota)))\circ(g(\iota),\varphi).
\]
The first factor is a product of elementary maps, one for each middle coordinate. Hence $F_0$ is tame. For $m=0$, this factor is the identity.

By Proposition~\ref{prop:ring-nilpotent}, $\Theta\in\GA_3(R)$ is stably tame over $R=K[\zeta]$. Thus $(\Theta,\psi_4,\ldots,\psi_{3+s})\in\TA_{3+s}(R)$ for some $s\geq0$. An elementary factor over $R$ is elementary over $K$ after the base variables $\zeta$ are included and held fixed. To treat the linear factors, use Suslin's theorem in the form
\begin{equation}\label{eq:Suslin}
 \GL_l(K[\zeta])=\langle E_l(K[\zeta]),\GL_l(K)\rangle,
 \qquad l\geq3;
\end{equation}
see~\cite[Theorem~3.23]{BEW}. Each elementary matrix in~\eqref{eq:Suslin} lifts to an elementary polynomial automorphism over $K$, while a matrix in $\GL_l(K)$ is linear over $K$. It follows that $\Psi(\zeta,\psi)=(\zeta,\Theta_{\zeta}(\psi))$ is stably tame over $K$.

Finally, $F=(F_0(\zeta),\psi)\circ\Psi$. The first factor is tame and the second is stably tame. Their composition is tame after a common stabilization. This also establishes polynomial invertibility. Explicitly, for a target $(\varpi,z)$ one first sets $\zeta=F_0^{-1}(\varpi)$ and then $\psi=\Theta_{\zeta}^{-1}(z)$; the coefficients of $\Theta^{-1}$ lie in $R$ by Proposition~\ref{prop:ring-nilpotent}.
\end{proof}

\begin{remark}\label{rem:block-scope}
Normalization of $\widetilde H$ is not needed in Theorem~\ref{thm3}, because Theorem~\ref{thm1} allows arbitrary constant terms and Proposition~\ref{prop:ring-nilpotent} explicitly subtracts the relative constant term. The argument gives no uniform bound on the number of stabilizing variables. It also does not assert tameness of the full map without stabilization.
\end{remark}

\section{Concluding remarks}\label{sec:conclusion}

The three-dimensional argument consists of a geometric reduction followed by a polynomial calculation. The generic-fiber construction uses the trace and principal-minor identities to obtain a plane equation and a quadratic equation. Homogeneous triangularization then restricts the directions at infinity. Once a dependent pair is available, the normal form and its tame factorization follow from the derivation argument of Proposition~\ref{prop:pair-normal}. The descent uses the one-dimensional space of quadratic relations, so the final conjugating maps are defined over the original field.

The canonical family also gives several explicit consequences. If $d=\deg P\geq1$, then
\[
 \deg\Ncal_P=4d,\qquad \rank J\Ncal_P=2.
\]
The minor in rows $1,2$ and columns $2,3$ of~\eqref{eq:canonical-J} is $P'(y+x^2)\neq0$. The $(1,3)$ entry of $(J\Ncal_P)^2$ is the same nonzero polynomial. Thus the nilpotency index is exactly three. The classification therefore implies that every normalized map in dimension three with nilpotent Jacobian and linearly independent components has degree divisible by four. This is a consequence of the proposed general theorem, whereas the displayed degree, rank, and nilpotency-index calculations for the canonical family are direct.

For the block extensions, the distinction between coefficients and polynomial variables is essential. The last three components must first be normalized over the fraction field of the base ring. The resulting linear normal form need not arise from a change of coordinates over the polynomial ring. Stable tameness follows from a separate argument: polynomial inversion descends through the formal inverse at the origin, and residue-field tameness permits the application of the regular-ring criterion. This argument does not require descent of the rational conjugating matrices.

The recent work of Casta{\~n}eda, Honorato, and Valenzuela-Henr{\'i}quez~\cite[Theorems~1.2, 5.1, and~8.1]{CHV2026} connects Keller maps with the weak Markus--Yamabe injectivity problem. Starting from a real Keller map $F=\id+H$ with $H(0)=JH(0)=0$ and component degrees $d_i$, their chain realization constructs a polynomial vector field $X_{F,\beta}$ satisfying
\[
 N=\sum_{i=1}^{n}\max\{d_i,2\}-n,
 \qquad
 \operatorname{Spec}\bigl(JX_{F,\beta}(Z)\bigr)=\{-1\}
 \quad (Z\in\mathbb R^N).
\]
Its equilibria correspond polynomially to $F^{-1}(\beta)$. Component degrees $(4,6,7)$ yield a degree-seven vector field in dimension $14$ with three rational equilibria. A separate construction yields a degree-three vector field in dimension $18$. Adjoining negative identity vector fields extends the counterexamples to every dimension at least $14$.

For these fields, $R=\id+X$ has nilpotent Jacobian, whereas $\id-R=-X$ is noninjective. Thus nilpotence alone does not imply invertibility in arbitrary dimension. The constructions do not impose the variable dependence required by the block theorems of Section~\ref{sec:blocks}.

In dimension three, Theorem~\ref{thm1} gives injectivity for the constant-spectrum class $X=-\id+H$ with $JH$ nilpotent. This does not settle the general Hurwitz case, where the eigenvalues need only have negative real parts. The cited work leaves dimensions $3$ through $13$ unresolved.

Several questions lie beyond the scope of these arguments. In dimension four, can the nilpotence identities control the degree of the embedded generic curve when the Jacobian rank is three? Under which additional hypotheses can two independent constant linear combinations of the components be made algebraically dependent? For the block extensions, can one bound the required stabilization in terms of the dimension or degrees? Answering these questions requires additional arguments; the three-dimensional plane-conic calculation does not extend formally to a general higher-dimensional map.

\section*{Acknowledgements}
The authors acknowledge the use of ChatGPT-6 Astra to assist with brainstorming, mathematical development, and manuscript drafting. The authors are solely responsible for the final content, analysis, and conclusions.

\section*{Funding}
This work was jointly supported by the National Key R\&D Program of China under grant No.~2023YFA1009401 and the Natural Science Foundation of China under grant No.~12171324.\\ 

The second author is supported by the Scientific Research Fund of Hunan Provincial Education Department (Grant No. 25B0088), the NSF of China (Grant No. 12371020) and the Construct Program of the Key Discipline in Hunan Province.

\bibliographystyle{amsplain}
\bibliography{references}

@article{ArzhantsevPetravchuk,
 author={Arzhantsev, Ivan V. and Petravchuk, Anatoliy P.},
 title={Closed polynomials and saturated subalgebras of polynomial algebras},
 journal={Ukrainian Mathematical Journal}, volume={59}, number={12},
 pages={1783--1790}, year={2007},
 note={Preprint: Closed and irreducible polynomials in several variables, arXiv:math/0608157; lemma numbering refers to the preprint},
 url={https://arxiv.org/abs/math/0608157}
}

@article{BCW,
 author={Bass, Hyman and Connell, Edwin H. and Wright, David},
 title={The {Jacobian} conjecture: Reduction of degree and formal expansion of the inverse},
 journal={Bulletin of the American Mathematical Society (N.S.)},
 volume={7}, number={2}, pages={287--330}, year={1982}
}

@article{BEW,
 author={Berson, Joost and van den Essen, Arno and Wright, David},
 title={Stable tameness of two-dimensional polynomial automorphisms over a regular ring},
 journal={Advances in Mathematics}, volume={230}, number={4--6},
 pages={2176--2197}, year={2012}, doi={10.1016/j.aim.2012.04.017}
}

@article{CastanedaEssen,
 author={Casta{\~n}eda, {\'A}lvaro and van den Essen, Arno},
 title={A new class of nilpotent {Jacobians} in any dimension},
 journal={Journal of Algebra}, volume={566}, pages={283--301}, year={2021}
}

@article{ChamberlandEssen,
 author={Chamberland, Marc and van den Essen, Arno},
 title={Nilpotent {Jacobians} in dimension three},
 journal={Journal of Pure and Applied Algebra}, volume={205}, number={1},
 pages={146--155}, year={2006}
}

@article{deBondt2006,
 author={de Bondt, Michiel},
 title={Quasi-translations and counterexamples to the homogeneous dependence problem},
 journal={Proceedings of the American Mathematical Society},
 volume={134}, pages={2849--2856}, year={2006}
}

@article{deBondtEssen,
 author={de Bondt, Michiel and van den Essen, Arno},
 title={The {Jacobian} conjecture: Linear triangularization for homogeneous polynomial maps in dimension three},
 journal={Journal of Algebra}, volume={294}, number={1}, pages={294--306}, year={2005}
}

@article{deBondtYan2014,
 author={de Bondt, Michiel and Yan, Dan},
 title={Triangularization properties of power linear maps and the structural conjecture},
 journal={Annales Polonici Mathematici}, volume={112}, number={3},
 pages={247--266}, year={2014}, doi={10.4064/ap112-3-4},
 url={https://arxiv.org/abs/1302.6930}
}

@unpublished{HeYan,
 author={He, Yuan and Yan, Dan},
 title={The classification of some polynomial maps in dimension three},
 note={Preprint, arXiv:2609.13843, 2026.}
}

@article{Keller,
 author={Keller, Ott-Heinrich}, title={Ganze {Cremona}-Transformationen},
 journal={Monatshefte f{\"u}r Mathematik und Physik},
 volume={47}, pages={299--306}, year={1939}, doi={10.1007/BF01695502}
}

@book{EssenBook,
 author={van den Essen, Arno},
 title={Polynomial Automorphisms and the {Jacobian} Conjecture},
 series={Progress in Mathematics}, volume={190}, publisher={Birkh{\"a}user},
 address={Basel}, year={2000}
}

@article{Yagzhev,
 author={Yagzhev, A. V.}, title={On {Keller}'s problem},
 journal={Siberian Mathematical Journal}, volume={21}, pages={747--754}, year={1980}
}

@article{Yan2022,
 author={Yan, Dan}, title={Some polynomial maps with {Jacobian} rank two or three},
 journal={Algebra Colloquium}, volume={29}, number={2}, pages={341--360}, year={2022}
}

@article{YanBondt2019,
 author={Yan, Dan and de Bondt, Michiel},
 title={The classification of some polynomial maps with nilpotent {Jacobians}},
 journal={Linear Algebra and its Applications}, volume={565},
 pages={287--308}, year={2019}, url={https://arxiv.org/abs/1710.04210}
}

@article{YanTang,
 author={Yan, Dan and Tang, Guoping},
 title={Polynomial maps with nilpotent {Jacobians} in dimension three},
 journal={Linear Algebra and its Applications}, volume={489},
 pages={298--323}, year={2016}
}

@unpublished{ElHilany2025,
 author={El Hilany, Boulos},
 title={Around the topological classification problem of polynomial maps: A survey},
 year={2025},
 note={Preprint, arXiv:2501.03828, version 2},
 eprint={2501.03828}, archivePrefix={arXiv}, primaryClass={math.AG},
 doi={10.48550/arXiv.2501.03828},
 url={https://arxiv.org/abs/2501.03828}
}

@unpublished{LeeLi2024,
 author={Lee, Kyungyong and Li, Li},
 title={On the two-dimensional {Jacobian} conjecture: {Magnus}' formula revisited, {IV}},
 year={2024},
 note={Preprint, arXiv:2408.01279},
 eprint={2408.01279}, archivePrefix={arXiv}, primaryClass={math.AG},
 doi={10.48550/arXiv.2408.01279},
 url={https://arxiv.org/abs/2408.01279}
}

@unpublished{RamirezValqui2025,
 author={Ram{\'i}rez, Valeria and Valqui, Christian},
 title={The {Groebner} basis and solution set of a polynomial system related to the {Jacobian} conjecture},
 year={2025},
 note={Preprint, arXiv:2506.05697},
 eprint={2506.05697}, archivePrefix={arXiv}, primaryClass={math.AG},
 doi={10.48550/arXiv.2506.05697},
 url={https://arxiv.org/abs/2506.05697}
}

@unpublished{HamadaKatoKomiya2025,
 author={Hamada, Lucas and Kato, Kazuki and Komiya, Ryo},
 title={A {Tate} algebra version of the {Jacobian} conjecture},
 year={2025},
 note={Preprint, arXiv:2502.10769},
 eprint={2502.10769}, archivePrefix={arXiv}, primaryClass={math.AG},
 doi={10.48550/arXiv.2502.10769},
 url={https://arxiv.org/abs/2502.10769}
}

@unpublished{Gao2026,
 author={Gao, Shuhong},
 title={Counterexamples to the {Jacobian} conjecture in dimensions greater than two},
 year={2026},
 note={Preprint, arXiv:2608.00222},
 eprint={2608.00222}, archivePrefix={arXiv}, primaryClass={math.AG},
 url={https://arxiv.org/html/2608.00222v1}
}

@misc{Tao2026,
 author={Tao, Terence},
 title={A digestion of the {Jacobian} conjecture counterexample},
 year={2026}, month=jul,
 howpublished={What's new, expository article},
 note={July 21, 2026},
 url={https://terrytao.wordpress.com/2026/07/21/a-digestion-of-the-jacobian-conjecture-counterexample/}
}

@unpublished{CHV2026,
 author={Casta{\~n}eda, {\'A}lvaro and Honorato, Gerardo and Valenzuela-Henr{\'i}quez, Francisco},
 title={The weak {Markus--Yamabe} conjecture fails in dimension 14},
 year={2026},
 note={Preprint, arXiv:2608.05392},
 eprint={2608.05392}, archivePrefix={arXiv}, primaryClass={math.AG},
 url={https://arxiv.org/abs/2608.05392}
}
\end{document}